\documentclass[12pt]{article}
\usepackage{graphicx}
\usepackage{amsfonts}
\usepackage{amssymb,amsmath}
\usepackage{amsfonts}
\usepackage{amscd}
\usepackage{amsthm}
\usepackage{cite}
\usepackage{xcolor}
\numberwithin{equation}{section}
\newtheorem{thm}{Theorem:}[section]
\newtheorem{lem}{Lemma:}[section]

\newtheorem{Def}{Definition:}[section]

\usepackage{latexsym}
\usepackage{hyperref}
\usepackage{setspace}
\begin{document}
	\title{On the Oscillation of Three Dimensional Katugampola Fractional Delay Differential Systems}
	{\author{A. Kilicman\footnote{E.Mail: akilicman@yahoo.com (A. Kilicman )}\\ Department of Mathematics,\\ Universiti Putra Malaysia, 43400, UPM Serdang, Selangor, Malaysia \vspace{0.3in} \\  V. Sadhasivam\footnote{E.Mail: ovsadha@gmail.com (V. Sadhasivam)}, M. Deepa\footnote{E.Mail: mdeepa.maths@gmail.com (M. Deepa)} \ and \ N. Nagajothi\footnote{E.Mail: nagajothi006@gmail.com (N. Nagajothi)}  \\ Post Graduate and Research Department of Mathematics,\\ Thiruvalluvar Government Arts College (Affli. to Periyar University),\\ Rasipuram - 637 401, Namakkal Dt.\\
			Tamil Nadu, India}}
	\date{}
	\maketitle{}
	
\begin{abstract} 
	In this study, we consider the three dimensional $\alpha$-fractional nonlinear delay differential system of the form 
	\begin{eqnarray*}
		D^{\alpha}\left(u(t)\right)&=&p(t)g\left(v(\sigma(t))\right),\\ D^{\alpha}\left(v(t)\right)&=&-q(t)h\left(w(t))\right),\\ D^{\alpha}\left(w(t)\right)&=& r(t)f\left(u(\tau(t))\right),~ t \geq t_0,
	\end{eqnarray*}
	where $0 < \alpha \leq 1$, $D^{\alpha}$ denotes the Katugampola fractional derivative of order $\alpha$. We have established some new oscillation criteria of solutions of differential system by using generalized Riccati transformation and inequality technique. The obtained results are illustrated with suitable examples.
	\\ [1ex]
	\noindent{\bf 2010 Mathematics Subject Classification:} 34A08, 34A34, 34K11.\\
	{\bf Key words:} Oscillation, Nonlinear differential system, Delay differential system, $\alpha$- fractional derivative.
\end{abstract}

\section{Introduction}
Over the years, there are many developments on ordinary differential equations and partial differential equations involving global fractional derivatives of Caputo, Riemann-Liouville or Hadamard type, we refer the reader to the monographs [1, 8, 9, 12, 18, 23, 28, 33], the references cited therein.\\

Fractional differential equations emerge in several engineering and scientific branches as the mathematical modeling of systems and processes in the field of chemical processes, electrodynamics of computer medium, polymer rheology, mathematical biology, etc. The applications of fractional calculus to biomedical problems are done in the areas of membrane biophysics and polymer viscoelasticity, where the experimentally observed power law dynamics for current-voltage and stress-strain relationships are concisely captured by fractional order differential equations.\\

At the end of the nineteenth century, the qualitative analysis of nonlinear systems of integer order differential equations was initiated by Henry Poincare. After words, there has been significant development in the study of oscillation of integer order differential systems [3, 4, 10, 13, 19, 22, 26, 30, 31].\\

For application in physics, Vreeke and Sandquist [32] proposed the systems of differential equations
\begin{eqnarray*}
	\frac{dx_1(t)}{dt}&=& x_1(\gamma_1(1-x_2)+\gamma_2(1-x_3)),\\
	\frac{dx_2(t)}{dt}&=& \gamma_3(x_1-x_2),\\
	\frac{dx_3(t)}{dt}&=& \gamma_4(x_1-x_3),
\end{eqnarray*}
to describe the two temperature feedback nuclear reactor problem, where $x_1$ is normalized neutron density, $x_2$ and $x_3$ are normalized temperatures, $x_2$ being associated with fuel and $x_3$, with the moderator or coolant, $\gamma_3$ and $\gamma_4$ are positive heat transfer coefficients, $\gamma_1$ and $\gamma_2$ are normalized effective neutron lifetime parameters associated with the temperature feedbacks. The expression $\rho=\gamma_1(1-x_2)+\gamma_2(1-x_3)$ in the first equation is called reactivity and is a measure of multiplication factor of the neutrons in the fission reactor.\\

Although, the oscillation theory of classical differential systems is well established, see [7, 11, 24, 27], the growth of oscillation theory of nonlinear fractional differential system is notably little bit slow due to the occurrence of nonlocal behavior of fractional derivatives possessing weakly singular kernels.\\

In 2014, Khalil et al. introduced the conformable fractional derivative. This conformable fractional derivative seems to be a kind local derivative without memory, see [2, 5, 16]. An interesting application of the conformable fractional derivative in physics was discussed  in [20], where it has been used to formulate an Action Principle for particles under frictional forces.\\

Conformable fractional derivatives is quickly generalized by Katugampola, which has been referred to here as the Katugampola fractional derivative. Nowadays, many authors got interested in this type of derivatives for their many nice properties [6, 14, 15].\\

In [29], Spanikova et al. investigated the oscillatory properties of three-dimensional differential systems of neutral type
\begin{eqnarray*}
	[y_1(t)-a(t)y_1(g(t))]^{'}&=& p_1(t)f_1(y_2(h_2(t))),\\
	y_2^{'}(t)&=& p_2(t)f_2(y_3(h_3(t))),\\
	y_3^{'}(t)&=& -p_3(t)f_3(y_1(h_1(t))),~t \in [0,\infty).
\end{eqnarray*}
In [25], Sadhasivam et al. studied the existence of solutions of  three-dimensional fractional differential systems of the following form
\begin{eqnarray*}
	D^{\alpha}_{0^{+}}u(t)=f_{1}(t, v(t), v^{'}(t)), ~t \in (0,1),\\
	D^{\beta}_{0^{+}}v(t)=f_{2}(t, w(t), w^{'}(t)), ~t \in (0,1),\\
	D^{\gamma}_{0^{+}}w(t)=f_{3}(t, u(t), u^{'}(t)), ~t \in (0,1),
\end{eqnarray*}
together with the Neumann boundary condition
\begin{eqnarray*}
	u^{'}(0)=u^{'}(1)=0,~v^{'}(0)=v^{'}(1)=0,~w^{'}(0)=w^{'}(1)=0,
\end{eqnarray*}
where $D^{\alpha}_{0^{+}}, D^{\beta}_{0^{+}}, D^{\gamma}_{0^{+}}$ are the standard Caputo fractional derivatives, $1<\alpha,\beta,\gamma \leq 2$.\\
\hspace{0.2in} To the best of the authors knowledge, it seems that there has been no work done on the oscillation of $\alpha$-fractional nonlinear three dimensional differential systems.
Motivated by the above observation, we propose the following system of the form 
\begin{eqnarray}\label{e1.1}
D^{\alpha}\left(u(t)\right)&=& p(t)g\left(v(\sigma(t))\right),\nonumber
\\ D^{\alpha}\left(v(t)\right)&=& -q(t)h\left(w(t))\right),\\ D^{\alpha}\left(w(t)\right)&=& r(t)f\left(u(\tau(t))\right),~ t \geq t_0,\nonumber
\end{eqnarray}
where $0 < \alpha \leq 1$, $D^{\alpha}$ denotes the $\alpha$-fractional derivative of order $\alpha$ with respect to t.\\
Throughout this paper, we assume that the following conditions:\begin{enumerate}
	\item[$(A_1)$] $p(t) \in C^{2 \alpha}([t_0,\infty), \mathbb{R}^+)$,
	$q(t) \in C^{\alpha}([t_0,\infty), \mathbb{R}^+)$,~
	$r(t) \in C([t_0,\infty), \mathbb{R}^+)$,~p(t), $q(t)$ and
	$r(t)$ are not identically zero on any interval of the form $[T_0, \infty)$, where $T_0 \geq t_0$, q(t) and r(t) are positive and decreasing;\\
	\item[$(A_2)$] $g \in C^{\alpha}(\mathbb{R},\mathbb{R}), vg(v)>0, D^{\alpha}g(v) \geq l^{'}>0$, $h \in C^{\alpha}(\mathbb{R},\mathbb{R}), wh(w)>0, D^{\alpha}h(w) \geq m^{'}>0$,~
	$f \in C(\mathbb{R},\mathbb{R}), uf(u)>0$ and $\frac{f(u)}{u} \geq k>0$ for $u \neq 0$;\\
	\item[$(A_3)$] $\sigma(t) \leq t$, $\tau(t) \leq t$ with $D^{\alpha}\sigma(t) \geq l>0$ and satifsfies $\lim_{t \to \infty}\sigma(t)=\infty$, $\lim_{t \to \infty}\tau(t)=\infty$;\\
	\item[$(A_4)$] we will consider the cases:
	\begin{align*}
	\int_{t_0}^{\infty}s^{\alpha-1}\frac{1}{b(s)}ds= \infty, \int_{t_0}^{\infty}s^{\alpha-1}\frac{1}{a(s)}ds= \infty,
	\end{align*}
	where $b(t)= \frac{1}{q(t)}, a(t)= \frac{1}{p(t)}$ and $c(t)=l^2l^{'}m^{'}r(t)$, $a(t), b(t)$ and $c(t)$ are positive real valued continuous functions with $b(t)t^{1-\alpha}<1$.\end{enumerate}

As a solution of system (1.1), we mean that it is a vector valued function $(u(t), v(t), w(t)) \in C^{\alpha}([T_1,\infty),\mathbb{R})$, with $T_1 = min \left\{\tau(t_1),\sigma(t_1)\right\}$ for some $t_1 \geq t_0$ which has the property that $$b(t)D^{\alpha}\left(a(t)D^{\alpha}u(t)\right) \in C^{\alpha}([T_1,\infty),\mathbb{R})$$ and satisfies the system (1.1)  on $[T_1,\infty)$. In the sequel, it will be always assumed that solutions of system (1.1) exist on some half line $[T_1,\infty), T_1 > t_0$. We restrict our attention only to the nontrivial solutions of system (1.1), that is, the solutions $(u(t), v(t), w(t))$ such that $\sup\left\{|u(s)|+|v(s)|+|w(s)|,~ t \leq s < \infty\right\}>0$ for any $t \geq T_1$.\\

A proper solution $(u(t), v(t), w(t))$ of the system (1.1) will be called oscillatory if all the components are oscillatory, otherwise it will be called nonoscillatory. The system (1.1) is called oscillatory if all proper solutions are oscillatory.\\
The main goal of this paper is to present some new oscillation criteria for the system (1.1) by making use of generalized Riccati transformation and inequality technique. 

Bearing these ideas in mind the article is organized as follows. In section 2, we recall some concepts relative to the $\alpha$- fractional derivative. In section 3, we presents some new conditions for the oscillatory behavior of the solutions of system (1.1). Two illustrative examples are included in the final part of the paper to demonstrate the efficiency of new theorems.

\section{Preliminaries}
Before starting our analysis of (1.1), we have to explain the meaning of the operator $D^{\alpha}$. For the sake of completeness, let us provide the essentials of fractional calculus according to the Katugampola $\alpha$-fractional derivatives and integrals which are useful throughout this paper. We begin with the following definition.

\begin{Def}\cite{14}\label{d2.1}
	Let $y: [0, \infty)\to \mathbb{R}$ and $t>0$. Then the fractional derivative of y of order $\alpha$ is given by
	\begin{align}\label{e2.1}
	D^{\alpha}(y)(t):= \lim\limits_{\epsilon \to 0} \dfrac{y(te^{\epsilon t^{-\alpha}})-y(t)}{\epsilon} \hspace{0.2in} \mbox{for} \hspace{0.2in} t>0,
	\end{align}
	$\alpha \in (0,1]$. If y is $\alpha$-differentiable in some $(0,a)$, $a>0$, and $\lim\limits_{t \to 0^+}D^{\alpha}(y)(t)$ exists, then define
	\begin{align*}
	D^{\alpha}(y)(0):= \lim\limits_{t \to 0^+}D^{\alpha}(y)(t).
	\end{align*}
\end{Def}
{\bf $\alpha$-fractional derivative satisfies the following properties.}\cite{14}\\
Let $\alpha \in (0,1]$ and f, g be $\alpha$- differentiable at a point $t>0$ . Then\\
\begin{itemize}
	\item[$(p_1)$]  $D^{\alpha}(t^n)=nt^{n-\alpha}$ for all $n \in \mathbb{R}$.\\
	\item[$(p_2)$]  $D^{\alpha}(C)=0$ for all constant functions, $f(t)=C$.\\
	\item[$(p_3)$]  $D^{\alpha}(fg)=fD^{\alpha}(g)+gD^{\alpha}(f)$.\\
	\item[$(p_4)$]  $D^{\alpha}(\frac{f}{g})=\frac{gD^{\alpha}(f)-fD^{\alpha}(g)}{g^2}$.\\
	\item[$(p_5)$]  $D^{\alpha}(f \circ g)(t)=f^{'}(g(t))D^{\alpha}g(t)$, for $f$ is differentiable at $g(t)$.\\
	\item[$(p_6)$]  If $f$ is differentiable, then $D^{\alpha}(f)(t)=t^{1-\alpha}\frac{df}{dt}(t)$.
\end{itemize} 
\begin{Def}\cite{14}\label{d2.2}
	Let $a \geq 0$ and $t \geq a$. Also, let y be a function defined on $(a,t]$ and $\alpha \in \mathbb{R}$. Then, the $\alpha$-fractional integral of y is given by
	\begin{align}
	I^{\alpha}_a(y)(t):= \int_{a}^{t}\dfrac{y(x)}{x^{1-\alpha}}dx
	\end{align}
	if the Riemann improper integral exists.
\end{Def}
\section{Main Results}
In this section, we study oscillatory behavior of solutions of the system (1.1) under certain conditions. We next establish the following lemmas needed in our further discussion.
\begin{lem}\label{l3.1}
	If $(x(t),y(t), z(t))$ is a nonoscillatory solution of (1.1), then the component function $x(t)$ is always nonoscillatory.
\end{lem}
\begin{proof} 
	The proof follows from Lemma 2.1 in \cite{21}.
\end{proof}
The next lemma is the $\alpha$-fractional analogue of well known result of \cite{17}.
\begin{lem}\label{l3.2}
	Suppose that {\color{red}$(A_1)$ and $(A_4)$} holds. Then there exists a $t_1 \geq t_0$ such that either
	$(I)$  $u(t)>0, D^{\alpha}u(t)>0, D^{\alpha}(a(t)D^{\alpha}u(t))>0$ for $t\geq t_1$.\\
	or \\
	$(II)$ $u(t)>0, D^{\alpha}u(t)<0, D^{\alpha}(a(t)D^{\alpha}u(t))>0$ for $t\geq t_1$ holds.
\end{lem}
\begin{proof} 
	let $u(t)$ be an eventually positive solution of (1.1) on $(t_0, \infty)$. Now, system (1.1) can be reduced to the following nonlinear delay differential inequality
	\begin{align}\label{e3.1}
	D^{\alpha}\left(\frac{1}{q(t)}D^{\alpha}\left(\frac{1}{p(t)}D^{\alpha}u(t)\right)\right)+l^2l^{'}m^{'}r(t)f(u(\tau(\sigma(t)))) \leq0,~t \geq t_1,
	\end{align}which implies,
	\begin{align}\label{e3.2}
	D^{\alpha}\left(b(t)D^{\alpha}\left(a(t)D^{\alpha}u(t)\right)\right)+c(t)f(u(\tau(\sigma(t)))) \leq 0,~t \geq t_1.
	\end{align} From $(3.2)$, we get $D^{\alpha}(b(t)D^{\alpha}(a(t)D^{\alpha}u(t))) \leq 0$ for $t\geq t_0$. Then $b(t)D^{\alpha}(a(t)D^{\alpha}u(t))$ is decreasing on $(t_0, \infty)$.\\
	Suppose not, $D^{\alpha}(a(t)D^{\alpha}u(t)) \leq 0$ then $a(t)D^{\alpha}u(t)$ is decreasing and there exists a constant $k^{'}$ and $t_2 \geq t_0$ such that $b(t)D^{\alpha}(a(t)D^{\alpha}u(t)) \leq -k^{'}$ for $t \geq t_2$. Integrating from $t_2$ to t, we get
	\begin{align}\label{e3.3}
	a(t)D^{\alpha}u(t) \leq a(t_2)D^{\alpha}u(t_2)-k^{'} \int_{t_2}^{t}s^{\alpha-1}\dfrac{1}{b(s)}ds.
	\end{align}
	Letting $t \to \infty$, and using $(A_4)$, we get $a(t)D^{\alpha}u(t) \to - \infty$. Hence, there is an integer $t_3 \geq t_2$ such that $a(t)D^{\alpha}u(t) \leq a(t_3)D^{\alpha}u(t_3) <0$ for $t \geq t_3$. Integrating from $t_3$ to t, we get
	\begin{align}\label{e3.4}
	u(t) \leq u(t_3)+ a(t_3)D^{\alpha}u(t_3)\int_{t_3}^{t}s^{\alpha-1}\dfrac{1}{a(s)}ds.
	\end{align}
	As $t \to \infty$, $u(t) \to - \infty$ by $(A_4)$. Which gives a contradiction to $u(t)>0$. We conclude that $D^{\alpha}(a(t)D^{\alpha}u(t))>0$ and $a(t)D^{\alpha}u(t)$ is increasing and we are led to (I) or (II).
\end{proof}
The following notations are employed in the sequel.
\begin{align}\label{e3.5}
(A_{\alpha})_*:=\liminf\limits_{t\to \infty}t\int\limits_{t}^{\infty}s^{\alpha-1}A^{\alpha}(s)ds~and~(B_{\alpha})_*:=\liminf\limits_{t\to \infty}\frac{1}{t}\int\limits_{t_0}^{t}s^{\alpha+1}A_{\alpha}(s)ds,
\end{align}
where $A_{\alpha}(t)=\frac{k}{2}\frac{c(t)}{a(t)}\frac{\tau(\sigma(t))-T}{t}(\tau(\sigma(t)))^{\alpha}$.
\begin{align}\label{e3.6}
d:=\liminf\limits_{t\to \infty}tw(t)~and~D:=\limsup\limits_{t\to \infty}tw(t).
\end{align}
\begin{thm}\label{t3.1}
	Suppose that the assumptions $(A_1) - (A_4)$ hold. Assume also that 
	\begin{align}\label{e3.7}
	\int\limits_{t_2}^{\infty}c(s)(s-T)\tau(\sigma(s))ds=\infty,
	\end{align}
	there exists a positive function $\rho \in C^{\alpha}([0,\infty);\mathbb{R}_+)$ such that
	\begin{align}\label{e3.8}
	\limsup\limits_{t\to \infty}\int\limits_{t_0}^{t}\bigg(s^{\alpha-1}\rho(s)A_{\alpha}(s)-\frac{1}{4}\frac{(\rho^{'}(s))^2}{\rho(s)}s^{1-\alpha}b(s)\bigg)ds=\infty.
	\end{align}
	Then every solution of system $(1.1)$ is oscillatory.
\end{thm}
\begin{proof}
	Suppose that $(1.1)$ has a nonoscillatory solution $(u(t), v(t), w(t))$ on $[t_0,\infty)$. From Lemma 3.1, $u(t)$ is always nonoscillatory. Without loss of generality, we shall assume that $u(t)>0$, $u(\tau(t))>0$ and $u(\tau(\sigma(t)))>0$ for $t \geq T \geq t_0 $, since the similar argument holds also for $u(t)<0$ eventually.	Suppose that Case (I) of Lemma 3.2 holds for $t \geq t_1$.\\
	Define the generalized Riccati substitution
	\begin{align}\label{e3.9}
	w(t)= \rho(t) \dfrac{b(t)D^{\alpha}(a(t)D^{\alpha}u(t))}{a(t)D^{\alpha}u(t)} ,~t \geq t_1. 
	\end{align}
	Thus $w(t)>0$, differentiating $\alpha$ times with respect to $t$, using (3.2), and $(A_2)$, we have
	\begin{align}\label{e3.10}
	D^{\alpha}w(t)\leq \frac{D^{\alpha}\rho(t)}{\rho(t)}w(t)-\frac{k\rho(t)c(t)}{a(t)}\frac{u(\tau(\sigma(t)))}{D^{\alpha}u(t)}-\frac{1}{\rho(t)b(t)}w^{2}(t),~t \geq t_1. 
	\end{align}
	Now, let $z_1(t,T)=(t-T),~z_2(t,T)=\frac{(t-T)^2}{2}$ and define\\
	$U(t):=(t-T)t^{1-\alpha}u(t)-z_2(t,T)D^{\alpha}u(t)$.\\
	Then $U(T)=0$, and differentiating the above, we get
	\begin{eqnarray*}
		D^{\alpha}U(t)&=& t^{1-\alpha}\bigg(t^{1-\alpha}u(t)+(t-T)(1-\alpha)t^{-\alpha}u(t)+(t-T)t
		^{1-\alpha}u^{'}(t)\\ && -z_2^{'}(t,T)D^{\alpha}u(t)-z_2(t,T)(D^{\alpha}u(t))^{'}\bigg), 
	\end{eqnarray*}
	which implies
	\begin{align}
	U^{'}(t)\geq t^{1-\alpha}u(t)-z_2(t,T)(D^{\alpha}u(t))^{'}. 
	\end{align}
	By Taylor's Theorem, we have
	\begin{align*}
	\int_{T}^{t}s^{1-\alpha}u^{'}(s)ds = z_1(t,T)D^{\alpha}u(T) +\int_{T}^{t}z_1(t,s)(D^{\alpha}u(s))^{'}ds,
	\end{align*}
	since $D^{\alpha}(a(t)D^{\alpha}u(t))$ is decreasing, we get
	\begin{align*}
	t^{1-\alpha}u(t)\geq t^{1-\alpha}u(T)+z_1(t,T)D^{\alpha}u(T) +(D^{\alpha}u(t))^{'}\int_{T}^{t}z_1(t,s)ds.
	\end{align*}
	Thus $U^{'}(t)>0$ on $[T,\infty)$. From this we get $U(t)>0$ on $[T,\infty)$, which implies that
	\begin{align}
	\frac{u(t)}{D^{\alpha}u(t)}>\frac{z_2(t,T)}{(t-T)t^{1-\alpha}}=\frac{t-T}{2}t^{1-\alpha},~t\in [T,\infty).
	\end{align}
	Next, define $V(t):=D^{\alpha}u(t)-t(D^{\alpha}u(t))^{'} $.\\
	In view of the fact that $D^{\alpha}V(t)=-t^{2-\alpha}(D^{\alpha}u(t))^{''}$,\\ which implies $V^{'}(t)=-t^{2-\alpha}(D^{\alpha}u(t))^{''}>0$ for $t \in (T,\infty)$, therefore $V(t)$ is strictly increasing on $(T,\infty)$. We claim that there is a $t_1 \in [T,\infty)$ such that $V(t)>0$ on $[t_1,\infty)$. Suppose not, $V(t)<0$ on $[t_1,\infty)$. Hence,
	\begin{align*}
	D^{\alpha}\bigg(\frac{D^{\alpha}u(t)}{t}\bigg)=-\frac{t^{1-\alpha}}{t^2}(t(D^{\alpha}u(t))^{'}-D^{\alpha}u(t)),
	\end{align*}which implies that  
	\begin{align*}
	\bigg(\frac{D^{\alpha}u(t)}{t}\bigg)^{'}=-\frac{1}{t^2}V(t)>0,~t\in (t_1,\infty).
	\end{align*}
	Choose $t_2 \in (t_1,\infty)$, for $t \geq t_2$, $\tau(\sigma(t)) \geq \tau(\sigma(t_2))$. Since, $\frac{D^{\alpha}u(t)}{t}$ is strictly increasing, 
	\begin{align*}
	\frac{D^{\alpha}u(\tau(\sigma(t)))}{\tau(\sigma(t))}\geq \frac{D^{\alpha}u(\tau(\sigma(t_2)))}{\tau(\sigma(t_2))}:=m>0,~t\in (t_1,\infty),
	\end{align*}
	(3.12) implies that
	\begin{align}
	u(\tau(\sigma(t)))\geq \frac{t-T}{2}t^{1-\alpha}m\tau(\sigma(t)).
	\end{align}
	Now, integrating (3.2) from $t_2$ to t, using $(A_2)$ and (3.13), we have
	\begin{align*}
	\int\limits_{t_2}^{t}\bigg((b(s)D^{\alpha}(a(s)D^{\alpha}u(s)))^{'}+\frac{km}{2}c(s)(s-T)\tau(\sigma(s))\bigg)ds\leq0.
	\end{align*} Then
	\begin{align*} b(t_2)D^{\alpha}(a(t_2)D^{\alpha}u(t_2))\geq \frac{km}{2}\int\limits_{t_2}^{t}c(s)(s-T)\tau(\sigma(s))ds,
	\end{align*}which contradicts to (3.7). Hence $V(t)>0$ on $(t_1,\infty)$.	Accordingly,
	\begin{align*}
	t^{1-\alpha}\bigg(\frac{D^{\alpha}u(t)}{t}\bigg)^{'}=-\frac{t^{1-\alpha}}{t^2}(t(D^{\alpha}u(t))^{'}-D^{\alpha}u(t))=-\frac{t^{1-\alpha}}{t^2}V(t)<0,~t\in (t_1,\infty),
	\end{align*}
	which gives $t(D^{\alpha}u(t))^{'}<D^{\alpha}u(t)$. Then $\tau(\sigma(t)) \leq \tau(t) \leq t $,
	\begin{align}
	\frac{D^{\alpha}u(\tau(\sigma(t)))}{\tau(\sigma(t))}\geq \frac{D^{\alpha}u(t)}{t},
	\end{align}since $\frac{D^{\alpha}u(t)}{t}$ is strictly increasing. Using (3.14) and (3.12) in (3.10), we get 
	\begin{align}\label{e3.15}
	D^{\alpha}w(t)\leq \frac{D^{\alpha}\rho(t)}{\rho(t)}w(t)-\frac{k\rho(t)c(t)}{ta(t)}\frac{(\tau(\sigma(t)))^{\alpha}(\tau(\sigma(t))-T)}{2}-\frac{1}{\rho(t)b(t)}w^{2}(t),~t \geq t_1. 
	\end{align}Therefore
	\begin{align*}
	D^{\alpha}w(t)\leq -\frac{k\rho(t)c(t)}{ta(t)}\frac{(\tau(\sigma(t)))^{\alpha}(\tau(\sigma(t))-T)}{2}+\frac{1}{4}b(t)\frac{(D^{\alpha}\rho(t))^2}{\rho(t)}, 
	\end{align*}
	using (3.5) and $(p_6)$, the above inequality becomes
	\begin{align}\label{e3.16}
	w^{'}(t) \leq -t^{\alpha-1}\rho(t)A_{\alpha}(t)+\frac{1}{4}\frac{(\rho^{'}(t))^2}{\rho(t)}t^{1-\alpha}b(t).
	\end{align}
	Integrating both sides from $t_1$ to t, we get that
	\begin{align*}
	\int\limits_{t_1}^{t}\bigg(s^{\alpha-1}\rho(s)A_{\alpha}(s)-\frac{1}{4}\frac{(\rho^{'}(s))^2}{\rho(s)}s^{1-\alpha}b(s)\bigg)ds\leq w(t_1),
	\end{align*}
	which contradicts the hypothesis $(3.8)$.
\end{proof}
We now derive various oscillatory criteria using the earlier results.\\
We can generalize the Philos type kernel. Let us introduce a class of functions $\mathfrak{R}$. Let
\begin{align*}
D_0=\left\{(t,s):t>s\geq t_0\right\},~D=\left\{(t,s):t\geq s\geq t_0\right\}.
\end{align*}
The function $H \in C(D,\mathbb{R})$ is said to belong to the class $\mathfrak{R}$, if\\
$(T_1)$  $H(t,t)=0$ for $t \geq t_0,~ H(t,s)>0$ for $(t,s)\in D_0$.\\
$(T_2)$  H has a continuous nonpositive partial derivative on $D_0$ with respect to s such that $h(t,s)=\frac{\partial H}{\partial s}(t,s)+H(t,s)\frac{\rho^{'}(s)}{\rho(s)}$.
\begin{thm}\label{t3.2}
	Suppose that the assumptions $(A_1) - (A_4)$ hold. Assume also that 
	there exists a positive function $\rho \in C^{\alpha}([0,\infty);\mathbb{R}_+)$ such that
	\begin{align}\label{e3.18}
	\limsup\limits_{t\to \infty} \frac{1}{H(t,t_1)}\int\limits_{t_1}^{t}\bigg( H(t,s)s^{\alpha-1}\rho(s)A_{\alpha}(s)-\frac{1}{4}\frac{\rho(s)b(s)}{H(t,s)}s^{1-\alpha}h^2(t,s)\bigg)ds=\infty.
	\end{align}
	Then every solution of system $(1.1)$ is oscillatory.
\end{thm}
\begin{proof}
	Proceeding as in the proof of Theorem 3.1, consider (3.15), we obtain the inequality
	\begin{align}\label{e3.19}
	w^{'}(t) \leq \frac{\rho^{'}(t)}{\rho(t)}w(t) -t^{\alpha-1}\rho(t)A_{\alpha}(t)-\frac{t^{\alpha-1}}{\rho(t)b(t)}w^{2}(t).
	\end{align}
	Integrating (3.18) both sides from $t_1$ to t, we get that
	\begin{align*}
	&\int\limits_{t_1}^{t} H(t,s)s^{\alpha-1}\rho(s)A_{\alpha}(s)ds\\ &\leq \int\limits_{t_1}^{t}H(t,s)\frac{\rho^{'}(s)}{\rho(s)}w(s)ds-\int\limits_{t_1}^{t}H(t,s)w^{'}(s)ds -\int\limits_{t_1}^{t}H(t,s)\frac{s^{\alpha-1}}{\rho(s)b(s)}w^{2}(s)ds,\\
	&\leq H(t,t_1)w(t_1)+ \int\limits_{t_1}^{t}\bigg(\frac{\partial H}{\partial s}(t,s)+H(t,s)\frac{\rho^{'}(s)}{\rho(s)}\bigg)w(s)ds-\int\limits_{t_1}^{t}H(t,s)\frac{s^{\alpha-1}}{\rho(s)b(s)}w^{2}(s)ds,
	\end{align*}
	\begin{align*}
	&\leq H(t,t_1)w(t_1)+ \int\limits_{t_1}^{t}\bigg(cw(s)-H(t,s)\frac{s^{\alpha-1}}{\rho(s)b(s)}w^{2}(s)\bigg)ds,\\
	&\leq H(t,t_1)w(t_1)+ \int\limits_{t_1}^{t}\frac{1}{4}\frac{\rho(s)b(s)}{H(t,s)}s^{1-\alpha}h^2(t,s)ds.
	\end{align*}Therefore we conclude that
	\begin{align*}
	\int\limits_{t_1}^{t}\bigg( H(t,s)s^{\alpha-1}\rho(s)A_{\alpha}(s)-\frac{1}{4}\frac{\rho(s)b(s)}{H(t,s)}s^{1-\alpha}h^2(t,s)\bigg)ds\leq H(t,t_1)w(t_1).
	\end{align*}Since $0<H(t,s) \leq H(t,t_1)$ for $t>s>t_1$, we have $0<\frac{H(t,s)}{H(t,t_1)} \leq 1$, hence
	\begin{align*}
	\frac{1}{H(t,t_1)}\int\limits_{t_1}^{t}\bigg( H(t,s)s^{\alpha-1}\rho(s)A_{\alpha}(s)-\frac{1}{4}\frac{\rho(s)b(s)}{H(t,s)}s^{1-\alpha}h^2(t,s)\bigg)ds\leq w(t_1).
	\end{align*}Letting $t \to \infty$,
	\begin{align*}
	\limsup\limits_{t\to \infty} \frac{1}{H(t,t_1)}\int\limits_{t_1}^{t}\bigg( H(t,s)s^{\alpha-1}\rho(s)A_{\alpha}(s)-\frac{1}{4}\frac{\rho(s)b(s)}{H(t,s)}s^{1-\alpha}h^2(t,s)\bigg)ds\leq w(t_1).
	\end{align*}
	Therefore assumption $(3.17)$ is contradicted. Thus every solution of (1.1) oscillates.
\end{proof}
We immediately obtain the following oscillation result for (1.1).
\begin{thm}\label{t3.3}
	Suppose that the assumptions $(A_1) - (A_4)$ hold. Assume also that 
	there exists a positive function $\rho \in C^{\alpha}([0,\infty);\mathbb{R}_+)$ such that
	\begin{align}\label{e3.20}
	\limsup\limits_{t\to \infty} \frac{1}{H(t,t_1)}\int\limits_{t_1}^{t}\bigg(H(t,s)s^{\alpha-1}\rho(s)A_{\alpha}(s)-\frac{1}{4}\frac{H(t,s)(\rho^{'}(s))^2}{\rho(s)}s^{1-\alpha}b(s)\bigg)ds=\infty.
	\end{align}
	Then every solution of system $(1.1)$ is oscillatory.
\end{thm}
\begin{proof}
	Proceeding as in the proof of Theorem 3.1, multiplying inequality (3.16) by H(t,s) and integrating both sides from $t_1$ to t, we get that
	\begin{align*}
	\int\limits_{t_1}^{t}\bigg(H(t,s)s^{\alpha-1}\rho(s)A_{\alpha}(s)-\frac{1}{4}\frac{H(t,s)(\rho^{'}(s))^2}{\rho(s)}s^{1-\alpha}b(s)\bigg)ds&\leq \int\limits_{t_1}^{t}H(t,s)w^{'}(s)ds,\\
	&\leq H(t,t_1)w(t_1).
	\end{align*}Taking limsup as $t \to \infty$, and hence
	\begin{align*}
	\limsup\limits_{t\to \infty} \frac{1}{H(t,t_1)}\int\limits_{t_1}^{t}\bigg(H(t,s)s^{\alpha-1}\rho(s)A_{\alpha}(s)-\frac{1}{4}\frac{H(t,s)(\rho^{'}(s))^2}{\rho(s)}s^{1-\alpha}b(s)\bigg)ds\leq w(t_1),
	\end{align*}
	which contradicts the hypothesis $(3.19)$.
\end{proof}
Following theorem to be proved using the techniques employed in the above theorems.
\begin{thm}\label{t3.4}
	Suppose that the assumptions $(A_1) - (A_4)$ and (3.7) hold. Assume also that Case (I) of Lemma 3.2 holds, then
	\begin{align}\label{e3.21}
	(A_{\alpha})_* \leq d-t^{\alpha-1}\frac{1}{b(s)}d^2,
	\end{align}and
	\begin{align}\label{e3.22}
	(B_{\alpha})_* \leq D-D^2.
	\end{align}
\end{thm}
\begin{proof}
	Let u(t) be a nonoscillatory solution of (3.2) such that $u(t)>0$, $u(\tau(t))>0$ and $u(\tau(\sigma(t)))>0$ for $t \geq T >t_0$, consider the case (I) of Lemma 3.2 holds, u(t) satisfies the inequality $D^{\alpha}\left(b(t)D^{\alpha}\left(a(t)D^{\alpha}u(t)\right)\right)\leq 0,~t \in [T,\infty)$. \\ 
	Define Riccati transformation
	\begin{align}\label{e3.23}
	w(t)=\dfrac{b(t)D^{\alpha}(a(t)D^{\alpha}u(t))}{a(t)D^{\alpha}u(t)}.
	\end{align}
	Thus $w(t)>0$, differentiating $\alpha$ times with respect to t, using (3.2), and $(A_2)$, we have
	\begin{align}\label{e3.24}
	D^{\alpha}w(t)\leq -\frac{kc(t)}{a(t)}\frac{u(\tau(\sigma(t)))}{D^{\alpha}u(t)}-\frac{1}{b(t)}w^{2}(t). 
	\end{align}By using (3.14), (3.12), and (3.5), we obtain the above inequality
	\begin{align}\label{e3.25}
	w^{'}(t)+t^{\alpha-1}A_{\alpha}(t)+t^{\alpha-1}\frac{1}{b(t)}w^{2}(t)\leq  0.
	\end{align}
	Given that $A_{\alpha}(t)>0$ and $w(t)>0$, which gives $w^{'}(t)\leq  0$, and
	\begin{align}\label{e3.26}
	-b(t)(w^{'}(t)t^{1-\alpha}/w^{2}(t))>1.
	\end{align}Which yields that
	\begin{align}\label{e3.27}
	\bigg(\frac{1}{w(t)}\bigg)^{'}>t^{\alpha-1}\frac{1}{b(t)}.
	\end{align}Integrating the above inequality from $t_1$ to t, and denote  $t_1^{\alpha-1}\frac{1}{b(t_1)}=M$, we have
	\begin{align}\label{e3.28}
	M(t-t_1)w(t)<1.
	\end{align}This implies that
	\begin{align}\label{e3.29}
	\lim_{t \to \infty}w(t)=0,~\lim_{t \to \infty}tw(t)=0.
	\end{align}From (3.6) and (3.27),
	\begin{align}\label{e3.30}
	0<d<1~ and ~0<D<1.
	\end{align}
	Even though if d=0 and D=0, there is nothing to prove. Now, to claim (3.20). Integrating (3.24) from t to $\infty$ and use (3.28), we get
	\begin{align}\label{e3.31}
	w(t)\geq \int\limits_{t}^{\infty}s^{\alpha-1}A_{\alpha}(s)ds+\int\limits_{t}^{\infty}s^{\alpha-1}\frac{1}{b(s)}w^{2}(s)ds.
	\end{align}Multiplying by t and taking liminf as $t \to \infty$, by (3.28), $d \geq (A_{\alpha})_*$. For any arbitrary $\epsilon>0$ and sufficiently small, there exists a $t_2 \geq t_1$ as
	\begin{align}\label{e3.32}
	d-\epsilon<tw(t)<d+\epsilon~ and ~t\int\limits_{t}^{\infty}s^{\alpha-1}A_{\alpha}(s)ds \geq (A_{\alpha})_* - \epsilon,~t\geq t_2.
	\end{align}
	Again from (3.31),
	\begin{align}\label{e3.33}
	tw(t)&\geq t\int\limits_{t}^{\infty}s^{\alpha-1}A_{\alpha}(s)ds+t\int\limits_{t}^{\infty}s^{\alpha-1}\frac{1}{b(s)}w^{2}(s)ds\nonumber \\
	&\geq t\int\limits_{t}^{\infty}s^{\alpha-1}A_{\alpha}(s)ds+t^{\alpha}\frac{1}{b(t)}\int\limits_{t}^{\infty}\frac{(sw(s))^2}{s^2}ds\nonumber \\&\geq t\int\limits_{t}^{\infty}s^{\alpha-1}A_{\alpha}(s)ds+t^{\alpha}\frac{1}{b(t)}(d-\epsilon)^2\int\limits_{t}^{\infty}\frac{1}{s^2}ds\nonumber\\
	&= t\int\limits_{t}^{\infty}s^{\alpha-1}A_{\alpha}(s)ds+t^{\alpha-1}\frac{1}{b(t)}(d-\epsilon)^2.
	\end{align}
	Therefore from (3.31) and (3.32), $d\geq (A_{\alpha})_* - \epsilon+(d-\epsilon)^2$. Then
	\begin{align*}
	d\geq (A_{\alpha})_*+t^{\alpha-1}\frac{1}{b(t)}d^2,
	\end{align*}since $\epsilon$ is arbitrarily small. Next to prove that (3.21). Multiply (3.24) by $s^2$, integrating from $t_1$ to t, using integration by parts, we obtain
	\begin{align*}
	\int\limits_{t_1}^{t}s^{\alpha+1}A_{\alpha}(s)ds&\leq - \int\limits_{t_1}^{t}s^2w^{'}(s)ds-\int\limits_{t_1}^{t}s^{\alpha+1}\frac{1}{b(s)}w^{2}(s)ds\\
	&\leq - t^2w(t)+t_1^2w(t_1)+2\int\limits_{t_1}^{t}sw(s)ds-\int\limits_{t_1}^{t}s^{\alpha+1}\frac{1}{b(s)}w^{2}(s)ds,
	\end{align*}which implies that
	\begin{align}\label{e3.34}
	t^2w(t) \leq t_1^2w(t_1)-\int\limits_{t_1}^{t}s^{\alpha+1}A_{\alpha}(s)ds+\int\limits_{t_1}^{t}\bigg(2sw(s)-s^{\alpha+1}\frac{1}{b(s)}w^{2}(s)\bigg)ds.
	\end{align}Thus, we have
	\begin{align}\label{e3.35}
	tw(t) &\leq \frac{t_1^2w(t_1)}{t}-\frac{1}{t}\int\limits_{t_1}^{t}s^{\alpha+1}A_{\alpha}(s)ds+\frac{1}{t}\int\limits_{t_1}^{t}s^{1-\alpha}b(s)ds,\nonumber \\
	&\leq \frac{t_1^2w(t_1)}{t}-\frac{1}{t}\int\limits_{t_1}^{t}s^{\alpha+1}A_{\alpha}(s)ds+\frac{1}{t}t^{1-\alpha}b(t)\int\limits_{t_1}^{t}ds.
	\end{align}
	By $A_4$, (3.34) imply that
	\begin{align}\label{e3.36}
	tw(t)\leq \frac{t_1^2w(t_1)}{t}-\frac{1}{t}\int\limits_{t_1}^{t}s^{\alpha+1}A_{\alpha}(s)ds+\frac{1}{t}(t-t_1).
	\end{align}Thus
	\begin{align*}
	\limsup\limits_{t\to \infty}tw(t)\leq 1- 	\liminf\limits_{t\to \infty} \frac{1}{t}\int\limits_{t_1}^{t}s^{\alpha+1}A_{\alpha}(s)ds.
	\end{align*}Hence from (3.5), (3.6), $D \leq 1-(B_{\alpha})_*$.
	For any arbitrary $\epsilon>0$ and sufficiently small, there exists a $t_2 \geq t_1$ such that
	\begin{align}\label{e3.37}
	D-\epsilon<tw(t)<D+\epsilon~ and ~\frac{1}{t}\int\limits_{t_0}^{t}s^{\alpha+1}A_{\alpha}(s)ds > (B_{\alpha})_* - \epsilon,~t\geq t_2.
	\end{align}Now, from (3.33) and (3.36) we get
	\begin{align}\label{e3.38}
	D \leq -(B_{\alpha})_*+ \epsilon(D+\epsilon)(2-D+\epsilon),~t\geq t_2,
	\end{align}since $\epsilon$ is  arbitrarily small, we have
	\begin{align*}
	(B_{\alpha})_* \leq D-D^2,
	\end{align*}which proves (3.21).
\end{proof}
\begin{lem}\label{l3.3}
	Suppose that the assumptions $(A_1) - (A_4)$ and (3.7) hold. Assume also that Case (I) of Lemma 3.2 holds. If
	\begin{align}\label{e3.39}
	\int\limits_{t_2}^{\infty}\eta^{\alpha-1}\frac{1}{a(\eta)}\bigg(\int\limits_{\eta}^{\infty}\int\limits_{\mu}^{\infty}s^{\alpha-1}c(s)dsd\mu\bigg)d\eta=\infty.
	\end{align}
	Then $\lim\limits_{t\to \infty}u(t)=0$.
\end{lem}
\begin{proof} 
	We consider Case (II) of Lemma 3.2, $D^{\alpha}u(t)<0, D^{\alpha}(a(t)D^{\alpha}u(t))>0$ for $t\geq t_1$. Since u(t) is positive and decreasing, there exists a $\lim\limits_{t\to \infty}u(t)=d^{'}\geq0$. Suppose not, $d^{'}>0$. Given that $u(\tau(\sigma(t))) \leq \tau(t) \leq t$, then $u(\tau(\sigma(t))) \geq u(t)>d^{'}$ for $t\geq t_2\geq t_1$ sufficiently large, u(t) is decreasing. Integrating (3.2) from $t$ to $\infty$ and using $u(\tau(\sigma(t))) \geq d^{'}$, we have
	\begin{align*}
	\int_{t}^{\infty}\bigg(b(s)D^{\alpha}\left(a(s)D^{\alpha}u(s)\right)\bigg)^{'}ds &\leq -\int_{t}^{\infty}ks^{\alpha-1}c(s)u(\tau(\sigma(s)))ds,\\
	&\leq -kd^{'}\int_{t}^{\infty}s^{\alpha-1}c(s)ds,
	\end{align*}then,
	\begin{align*}
	b(t)D^{\alpha}\left(a(t)D^{\alpha}u(t)\right) \geq kd^{'}\int_{t}^{\infty}s^{\alpha-1}c(s)ds.
	\end{align*}
	By $A_4$, we get
	\begin{align*}
	\left(a(t)D^{\alpha}u(t)\right)^{'} &\geq kd^{'}\frac{1}{b(t)t^{1-\alpha}}\int_{t}^{\infty}s^{\alpha-1}c(s)ds,\\
	&\geq kd^{'}\int_{t}^{\infty}s^{\alpha-1}c(s)ds.
	\end{align*}
	Again integrating the above inequality from t to $\infty$, we obtain
	\begin{align*}
	-a(t)D^{\alpha}u(t) &\geq kd^{'}\int_{t}^{\infty} \int_{\mu}^{\infty}s^{\alpha-1}c(s)dsd\mu,
	\end{align*}then,
	\begin{align*}
	-u^{'}(t) &\geq kd^{'}t^{\alpha-1}\frac{1}{a(t)} \int_{t}^{\infty} \int_{\mu}^{\infty}s^{\alpha-1}c(s)dsd\mu.
	\end{align*}
	Once again integrating this inequality from $t_2$ to $\infty$, we get
	\begin{align*}
	u(t_2) &\geq kd^{'} \int_{t_2}^{\infty}\bigg(\eta^{\alpha-1}\frac{1}{a(\eta)} \int_{\eta}^{\infty} \int_{\mu}^{\infty}s^{\alpha-1}c(s)dsd\mu \bigg)d\eta.
	\end{align*}
	which contracticts to (3.38). Thus $d^{'}=0$ and hence $\lim\limits_{t\to \infty}u(t)=0$.
\end{proof}
From Theorem 3.4, Nehari type oscillation criteria for (1.1).
\begin{thm}\label{t3.5}
	Suppose that the assumptions $(A_1) - (A_4)$, (3.7) and (3.38) hold.If
	\begin{align}\label{e3.39}
	\liminf\limits_{t\to \infty}\frac{1}{t}\int\limits_{t_0}^{t}\bigg(ks^{\alpha+1}\frac{c(s)}{a(s)}\frac{u(\tau(\sigma(s)))-T}{s}(u(\tau(\sigma(s))))^{\alpha}\bigg)ds>\frac{1}{2},
	\end{align}then u(t) is oscillatory or $\lim\limits_{t\to \infty}u(t)=0$.
\end{thm}
\section{Examples}
In this section, we will present some examples to explain the effectiveness of our results.\\
\textbf{Example 1.}
Consider the $\alpha$-fractional delay differential system
\begin{align}\label{e4.1}
	D^{\frac{1}{2}}(u(t))&=\frac{1}{\sqrt{t}}g(v(\frac{t}{2}))\nonumber \\
	D^{\frac{1}{2}}(v(t))&=-\frac{1}{\sqrt{t}}h(w(t)),\\
	D^{\frac{1}{2}}(w(t))&=\frac{1}{\sqrt{t}}f(u(\frac{t}{2})),\nonumber ~~t \geq t_0, 
	\end{align} where $C_1=\cos(ln 2),~C_2=\sin (\ln 2),~ A_1=\cos (\ln 4),~ A_2=\sin (\ln 4)$.\\
Here $\alpha=\frac{1}{2}$, $p(t)=\frac{1}{a(t)}= \frac{1}{\sqrt{t}},~ q(t)=\frac{1}{b(t)}=\frac{1}{\sqrt{t}},~ r(t)=\frac{1}{\sqrt{t}}$, 
$f(u)=A_1\sqrt{(1-u^2)}-A_2u$, $g(v)=v$ and $h(w)=w$.\\
It is easy to see that $D^{\alpha}g(v)=\frac{1}{\sqrt{t}}(C_1+C_2) \geq l^{'}>0$, $D^{\alpha}h(w)=\frac{1}{\sqrt{t}}(C_1-C_2)\geq m^{'}>0$, $f(u)/u=A_1\sqrt{\frac{1}{u^2}-1}-A_2\geq 0.2579=k>0$, since $u^2 <1$, $\sigma(t)=\tau(t)=\frac{t}{2}$ and $D^{\alpha}\sigma(t)=\frac{\sqrt{t}}{2} \geq l>0$, $c(t)= \frac{C_1^2 - C_2^2}{4 \sqrt{t}}$, $A_{\alpha}(t)=\frac{0.2579}{16} \frac{\frac{t}{4}-T}{\sqrt{t}}$.\\
Now consider,
\begin{align*}
\int\limits_{t_2}^{\infty}c(s)(s-T)\tau(\sigma(s))ds=
\frac{C_1^2 - C_2^2}{4}\int\limits_{t_2}^{\infty} \frac{(s-T)}{\sqrt{s}}\frac{s}{4}ds \to \infty~ as~ t \to \infty.
\end{align*}
If we take $\rho(t)=16/k$ then $\rho^{'}(t)=0$. Consider
\begin{align*}
\limsup\limits_{t\to \infty}\int\limits_{t_1}^{t}&\bigg(s^{\alpha-1}\rho(s)A_{\alpha}(s)-\frac{1}{4}\frac{(\rho^{'}(s))^2}{\rho(s)}s^{1-\alpha}b(s)\bigg)ds\\
&=\limsup\limits_{t\to \infty}\int\limits_{t_1}^{t}\bigg(s^{-\frac{1}{2}}\frac{16}{k}\frac{k(\frac{s}{4}-T)}{16}\frac{1}{\sqrt{s}}\bigg)ds\\
&= \limsup\limits_{t\to \infty}\frac{1}{4}\int\limits_{t_1}^{t}\bigg(\frac{s-4T}{s}\bigg)ds
\to \infty~ as~ t \to \infty.
\end{align*}
All the conditions of Theorem 3.1 are satisfied. Hence every solution of $(4.1)$ is oscillatory. Thus $(u(t), v(t), w(t))=(\sin (ln t), c_1 \cos(ln t)-C_2 \sin (ln t), C_1 \sin (ln t)+c_2 \cos (ln t))$ is one such solution.\\
\textbf{Note:} The decreasing condition imposed on q(t) and r(t) is only a sufficient condition, however it is not a necessary one. The following example ensures the oscillatory behavior of the system (4.2) even though q(t) and r(t) is nondecreasing.\\
\textbf{Example 2.}
Consider the system of $\alpha$-fractional differential equations
	\begin{align}\label{e4.2}
	D^{\frac{1}{3}}(u(t))&=\frac{t^{\frac{2}{3}}}{1+\frac{3}{4}\cos^{\frac{5}{3}}(t)}g(v(t-2\pi)), \nonumber \\
	D^{\frac{1}{3}}(v(t))&=-t^{\frac{2}{3}}w(t),\\
	D^{\frac{1}{3}}(w(t))&=\frac{t^{\frac{2}{3}}}{1+\cos^{2}(t)}f(u(t-\frac{3\pi}{2})),\nonumber ~~t \geq t_0.
	\end{align}
Here $\alpha=\frac{1}{3}$, $p(t)=\frac{1}{a(t)}= \frac{t^{\frac{2}{3}}}{1+\frac{3}{4}\cos^{\frac{5}{3}}(t)},~ q(t)=\frac{1}{b(t)}=t^{\frac{2}{3}},~ r(t)=\frac{t^{\frac{2}{3}}}{1+\cos^{2}(t)}$, 
$f(u)=u(1+u^2)$, $g(v)=v(1+\frac{3}{4}v^{\frac{5}{3}})$ and $h(w)=w$.\\
It is easy to see that $D^{\alpha}g(v)=v^{\frac{1}{3}}+2v^2 \geq 1 = l^{'}>0$ such that $y^2>1,~y^{'}>\frac{1}{3}$, $D^{\alpha}h(w)\geq 1=m^{'}>0$, $f(u)/u=1+u^2\geq 1=k>0$, $\sigma(t)=t-2\pi$, $\tau(t)=t-\frac{3\pi}{2}$ and $D^{\alpha}\sigma(t)=t^{\frac{2}{3}} \geq l$ such that $t_1=l^{\frac{2}{3}}$ for $t \geq t_1$, $c(t)=l^2 \frac{t^{\frac{2}{3}}}{1+\cos^{2}(t)}$,\\ $$A_{\alpha}(t)=\frac{l^2}{2}\frac{1+\frac{3}{4}\cos^{\frac{5}{3}}(t)}{1+\cos^{2}(t)}\frac{t-\frac{3\pi}{2}-T}{t}\left(t-\frac{3\pi}{2}\right)^{\frac{1}{3}}.$$
Now consider,
\begin{align*}
\int\limits_{t_2}^{\infty}c(s)(s-T)\tau(\sigma(s))ds&=
\int\limits_{t_2}^{\infty}l^2 \frac{s^{\frac{2}{3}}}{1+\cos^{2}(s)}(s-T)(s-\frac{3\pi}{2})ds\\ & \geq \frac{l^2}{2} \int\limits_{t_2}^{\infty} s^{\frac{2}{3}}(s-T)(s-\frac{3\pi}{2})ds \to \infty~ as~ t \to \infty.
\end{align*}
If we take $\rho(t)=1$ then $\rho^{'}(t)=0$. Consider
\begin{align*}
\limsup\limits_{t\to \infty}\int\limits_{t_1}^{t}&\bigg(s^{\alpha-1}\rho(s)A_{\alpha}(s)-\frac{1}{4}\frac{(\rho^{'}(s))^2}{\rho(s)}s^{1-\alpha}b(s)\bigg)ds\\
&=\limsup\limits_{t\to \infty}\int\limits_{t_1}^{t}\bigg(s^{-\frac{2}{3}}l^2\frac{s^{\frac{2}{3}}}{1+\cos^{2}(s)}\frac{1}{2}\frac{1+\frac{3}{4}\cos^{\frac{5}{3}}(s)}{1+\cos^{2}(s)}\frac{s-\frac{3\pi}{2}-T}{s}(s-\frac{3\pi}{2})^{\frac{1}{3}}\bigg)ds\\
&\geq \limsup\limits_{t\to \infty}\frac{7l^2}{16}\int\limits_{t_1}^{t}\bigg((1-\frac{\frac{3\pi}{2}-T}{s})(s-\frac{3\pi}{2})^{\frac{1}{3}}\bigg)ds
\to \infty~ as~ t \to \infty.
\end{align*}
All the conditions of Theorem 3.1 are satisfied. Hence every solution of $(4.2)$ is oscillatory. Thus $(u(t), v(t), w(t))=(\sin t,\cos t, \sin t)$ is one such solution.\\
\textbf{Example 3.}
Consider the $\alpha$-fractional differential system
\begin{align}\label{e4.3}
	D^{\frac{1}{2}}(u(t))&=e^{2t}t^{\frac{1}{2}}g(v(t-1)),\nonumber \\
	D^{\frac{1}{2}}(v(t))&=-e^{-2t}t^{\frac{1}{2}}w(t),\\
	D^{\frac{1}{2}}(w(t))&=(et)^{\frac{1}{2}}f(u(t-\frac{1}{2})),\nonumber ~~t \geq t_0.
	\end{align}
Here $\alpha=\frac{1}{2}$, $p(t)=\frac{1}{a(t)}= e^{2t}t^{\frac{1}{2}},~ q(t)=\frac{1}{b(t)}=e^{-2t}t^{\frac{1}{2}},~ r(t)=(et)^{\frac{1}{2}}$, 
$f(u)=u$, $g(v)=v$ and $h(w)=w$.\\
It is easy to see that $D^{\alpha}g(v)=v^{\frac{1}{2}}=e^{-\frac{t}{2}}=l^{'}>0$, $D^{\alpha}h(w)=w^{\frac{1}{2}}=e^{\frac{t}{2}}=m^{'}>0$, $f(u)/u=1=k>0$, $\sigma(t)=t-1$, $\tau(t)=t-\frac{1}{2}$ and $D^{\alpha}\sigma(t)=t^{\frac{1}{2}} \geq l$ such that $t_1=l^{\frac{1}{2}}$ for $t \geq t_1$, $c(t)=l^2 (et)^{\frac{1}{2}}$, $$A_{\alpha}(t)=\frac{l^2}{2}e^{\frac{1}{2}}e^{2t}(t-\frac{1}{2}-T)(t-\frac{1}{2})^{\frac{1}{2}}.$$
Now,
\begin{align*}
\int\limits_{t_2}^{\infty}c(s)(s-T)\tau(\sigma(s))ds&=
\int\limits_{t_2}^{\infty}l^2 (es)^{\frac{1}{2}}(s-T)(s-\frac{1}{2})ds\\ & = l^2 e^{\frac{1}{2}} \int\limits_{t_2}^{\infty} s^{\frac{1}{2}}(s-T)(s-\frac{1}{2})ds \to \infty~ as~ t \to \infty.
\end{align*}
If we take $\rho(t)=\frac{1}{t^{\frac{7}{2}}e^{2t}}$ then $\rho^{'}(t)=-\frac{t^{\frac{7}{2}}}{2t^7e^{2t}}(4t+7)$. Consider
\begin{align*}
\limsup\limits_{t\to \infty}\int\limits_{t_1}^{t}&\bigg(s^{\alpha-1}\rho(s)A_{\alpha}(s)-\frac{1}{4}\frac{(\rho^{'}(s))^2}{\rho(s)}s^{1-\alpha}b(s)\bigg)ds\\
&=\limsup\limits_{t\to \infty}\int\limits_{t_1}^{t}\bigg(s^{-\frac{1}{2}}\frac{l^2e^{\frac{1}{2}}}{s^{\frac{7}{2}}e^{2s}}\frac{1}{2}e^{2s}(s-\frac{1}{2}-T)(s-\frac{1}{2})^{\frac{1}{2}}-\frac{(4s+7)^2}{4s^9e^{4s}}e^{4s}s^{\frac{7}{2}}\bigg)ds
\end{align*}
\begin{align*}
&\leq \limsup\limits_{t\to \infty}\int\limits_{t_1}^{t}\bigg(\frac{l^2 e^{\frac{1}{2}}}{2s^{\frac{7}{2}}}s-\frac{s^2}{4s^{\frac{11}{2}}}\bigg)ds\\
&\leq \limsup\limits_{t\to \infty}\int\limits_{t_1}^{t}\bigg(\frac{l^2 e^{\frac{1}{2}}}{2s^{\frac{5}{2}}}-\frac{1}{4s^{\frac{7}{2}}}\bigg)ds
< \infty.
\end{align*}
The condition (3.8) of Theorem 3.1 is not satisfied, in view of the fact that $(A_4)$ fails to hold, and hence the system (4.3) is not oscillatory. In fact, $(u(t), v(t), w(t))=(e^t,e^{-t},e^t)$ is such a solution of (4.3), which is a nonoscillatory solution.

\section*{Conclusion}

\hspace{0.2in} In this study, the authors have derived some new oscillation results for some class of nonlinear three-dimensional $\alpha$-fractional differential system by using Riccati transformation and inequality technique. This work extends some of the known results in the classical literature to $\alpha$- fractional systems.

\end{document}